\documentclass{amsart}
\usepackage{hyperref}
\usepackage{enumitem}
\usepackage{graphicx}
\usepackage{amssymb}
\usepackage{upgreek}
\usepackage{comment}
\usepackage{amsmath}
\usepackage{amsfonts}
\usepackage{amstext}
\usepackage{amsthm} 
\usepackage{marginnote}
\usepackage{setspace}
\usepackage{array}
\usepackage{textcomp}
\usepackage{tipa}
\usepackage{stmaryrd}
\usepackage{graphicx}

\usepackage{caption}
\usepackage{subcaption}

\newcommand{\be}{\begin{equation} }
\newcommand{\ee}{\end{equation} }
\newcommand{\bee}{\begin{eqnarray} }
\newcommand{\eee}{\end{eqnarray} }

\newcommand{\pa}{\partial}

\newcommand{\la}{\lambda}

\renewcommand{\Re}{{\operatorname{Re}\,}}

\numberwithin{equation}{section}

\newtheorem{Lem}{Lemma}
\newtheorem{Thm}{Theorem}
\newtheorem*{claim}{Claim}

\newtheorem{remark}{Remark}

\newtheorem{corollary}{Corollary}
\newtheorem{proposition}{Proposition}

\newcommand{\set}[1]{\ensuremath{\{#1\}}}

\newcommand{\EE}[1]{\mathbb E}

\newcommand{\leb}{\ensuremath{\lambda}}
\newcommand{\ep}{\ensuremath{\epsilon}}
\newcommand{\eps}{\ensuremath{\epsilon}}
\newcommand{\R}{\ensuremath{\mathbb R}}

\newcommand{\C}{\ensuremath{\mathbb C}}
\newcommand{\w}{\ensuremath{\omega}}
\newcommand{\W}{\ensuremath{\Omega}}

\newcommand{\setst}[2]{\ensuremath{\left\{#1\,\middle|\,#2\right\}}}

\newcommand{\abs}[1]{\left\lvert #1 \right\rvert}
\newcommand{\norm}[1]{\left\lVert#1\right\rVert}

\newcommand{\gives}{\ensuremath{\rightarrow}}

\newcommand{\x}{\ensuremath{\times}}

\newcommand{\lr}[1]{\ensuremath{\left(#1\right)}}
\renewcommand{\Re}{\ensuremath{\mathrm{Re} }}

\newcommand{\dell}{\ensuremath{\partial}}

\newcommand{\Z}{{\mathbb Z}}

\newcommand{\hcal}{\mathcal H}
\newcommand{\Rcal}{\mathcal R}

\newcommand{\lbr}{\llbracket}
\newcommand{\rbr}{\rrbracket}

\usepackage{color}

\begin{document}
\author{Thomas Beck, Spencer Becker-Kahn, Boris Hanin}

\address{Department of Mathematics, MIT, Cambridge, MA 02139, USA}
\email[B. Hanin]{bhanin@mit.edu}
\email[T. Beck]{tdbeck@mit.edu}
\email[S. T. Becker-Kahn]{stbeckerkahn@mit.edu}

\thanks{BH is partially supported by NSF grant DMS-1400822.}

\begin{abstract}
We show that on a compact Riemmanian manifold $(M,g)$, nodal sets of linear combinations of any $p+1$ smooth functions form an admissible $p-$sweepout provided these linear combinations have uniformly bounded vanishing order. This applies in particular to finite linear combinations of Laplace eigenfunctions. As a result, we obtain a new proof of the Gromov, Guth, Marques--Neves upper bounds on the min-max $p$-widths of $M.$ We also prove that close to a point at which a smooth function on $\mathbb{R}^{n+1}$ vanishes to order $k$, its nodal set is contained in the union of $k$ $W^{1,p}$ graphs for some $p > 1$. This implies that the nodal set is locally countably $n$-rectifiable and has locally finite $\mathcal{H}^n$ measure, facts which also follow from a previous result of B\"ar. Finally, we prove the continuity of the Hausdorff measure of nodal sets under heat flow. 
\end{abstract}

\title[Smooth Functions with Finite Vanishing Order]{Nodal Sets of Smooth Functions with Finite Vanishing Order and $p$-Sweepouts}
\setcounter{section}{0}
\maketitle

\section{Introduction}
This article concerns the regularity of nodal sets in families of smooth functions with finite vanishing order. Our motivation comes in part from the work of Marques-Neves \cite{marquesnevesexistence}, who use admissible $p$-sweepouts in a compact Riemmanian manifold $(M,g)$ to prove the existence of infinitely many closed minimal hypersurfaces if $M$ has positive Ricci curvature. Each admissible $p$-sweepout is essentially a $p$-dimensional family of co-dimension $1$ cycles in $M$ (see \S \ref{SS:Psweepout}), and the associated min-max $p$-widths $\w_p(M)$ (see Theorem \ref{T:UpperBound}) can be thought of as giving a non-linear version of the spectrum of the Laplacian. An analogy like this was first proposed by Gromov in \cite{gromovdimension}.

Marques-Neves suggested in \cite[\S 9]{marquesnevesexistence} that one might push this analogy further by considering $p$-sweepouts formed from the nodal sets of linear combinations of Laplace eigenfunctions. However, it was not clear at the time that a $p$-dimensional family of cycles defined in this way would satisfy the technical conditions needed to be admissible as a $p$-sweepout. In Theorem \ref{T:Psweepout} we provide a general construction of admissible $p-$sweepouts from the nodal sets of families of smooth functions that have uniformly bounded vanishing order. Our construction applies in particular to finite linear combinations of eigenfunctions. Theorem \ref{T:Psweepout} also yields a new proof of the Weyl-type upper bounds on the $p-$widths $\w_p(M).$



To view a family of nodal sets as an admissible $p$-sweepout, one must control the extent to which the nodal sets can concentrate in small balls in $M$ (see \S \ref{SS:Psweepout}). Estimates that provide this control follow both from the new general $W^{1,p}$ regularity result that we present here, Theorem \ref{T:Main}, and from previous work of B\"ar \cite{barzero} (see Proposition \ref{P:Bar}). Both Theorem \ref{T:Main} and Proposition \ref{P:Bar} imply that near a point of finite vanishing order, the nodal set of a smooth function on $\mathbb{R}^{n+1}$ is countably $n$-rectifiable and has locally finite $\mathcal{H}^n$ measure (see \S \ref{S:W1p} and \S \ref{S:nonconc}). They also allow us to study in \S \ref{S:Heat} the evolution of nodal sets for families 
\[\set{e^{-t\Delta_g}u}_{t\geq 0},\qquad u\in L^2(M,g)\]
under heat flow. 


\subsection{Regularity and measure of nodal sets for families of smooth functions}
By a result of Whitney \cite{whitneyanalytic}, every closed subset of $\R^{n+1}$ is the nodal set $Z_f=f^{-1}(0)$ of some smooth real-valued function $f$. This means that, in general, $Z_f$ can be arbitrarily irregular. Constraints on the derivatives of $f$ restrict the possible behavior of $Z_f,$ however. For example, if $f(x)=0$ and $\nabla f(x)\neq 0,$ then, by the implicit function theorem, $Z_f$ is a smooth manifold near $x$. 

Solutions of elliptic or parabolic PDEs satisfy more sophisticated constraints that allow for quantitative estimates on Hausdorff measures of nodal and singular sets. Early results in this setting are due to Carleman \cite{carlemanunicite}, who established finite vanishing order for solutions to second order elliptic equations. His method strongly influenced later work. Further developments of particular note include the work of Garafalo--Lin \cite{garofalolinmonotonicity, garofalolinunique} on elliptic equations and Lin \cite{linuniqueness} for parabolic equations, with the strongest quantitative results by Hardt-Simon \cite{hardtsimonnodal}, Donnelly-Fefferman \cite{donfeffnodal}, Naber-Valtorta \cite{nabervaltortasingular}, and recently Logunov \cite{logunovupper, logunovnadirashvili}, Logunov-Malinnikova \cite{logmallowdim}. 

We are concerned here, however, with what can be about $Z_f$ if $f$ vanishes to finite order but does not necessarily satisfy a PDE. Lin showed in \cite{linnodal} that such functions include finite linear combinations of Laplace eigenfunctions (alternative proofs were given by Donnelly \cite{donnellynodalsums} and Jerison--Lebeau \cite{jerisonlebeaunodal}). Jerison--Kenig \cite{jerisonkenigunique} also obtained similar statements about solutions to certain differential inequalities.

Throughout, $f$ is a smooth function. Therefore, it has finite vanishing order in an open set $U\subseteq \R^{n+1}$ if for each $x \in U$ there exists a multi-index $\alpha$ for which $D^{\alpha} f(x) \neq 0$. If $|\alpha| = \gamma$, and $D^{\beta}f(x) = 0$ for all multi-indices $\beta$ with $|\beta|<\gamma$, then $f$ is said to have vanishing order $\gamma$ at $x$. The following was shown by B\"{a}r:

\begin{proposition}[\cite{barzero}, Lemma $3$] \label{P:Bar} Let $f : \R^{n+1} \to \R$ be smooth and suppose that $f$ vanishes to order $\gamma$ at $x_0 \in \R^{n+1}$. Then there is $\bar{r} > 0$ and a hyperplane $P\subseteq \R^{n+1}$ such that $Z_f \cap B_{\bar{r}}(x_0)$ is contained in the union of countably many graphs of smooth real-valued functions from $P\cap B_{\bar{r}}(x_0)$ to $P^\perp$. Moreover, we can estimate the Hausdorff measure of the nodal set by
\[
\hcal^n(Z_f \cap B_r(x_0)) \leq (n+1)2^{n}\gamma r^n
\] 
for all $r < \bar{r}$.
\end{proposition}



The radius $\bar{r}$ in Proposition \ref{P:Bar} can be chosen uniformly over $\theta\in \Theta$ for families
\begin{equation}
f^\theta(x):=F(x,\theta)\label{E:Family}
\end{equation}
where $F\in C^\infty(U\x \Theta)$, the set $\Theta$ is a finite-dimensional compact smooth manifold (possibly with boundary), and $U\subseteq \R^{n+1}$ is open. Denoting by $\Gamma_u$ the graph of a function $u,$ we obtain the following regularity result. 





\begin{Thm}\label{T:Main}
Let $f^\theta$ be as in \eqref{E:Family} and suppose that the vanishing order of $f^{\theta_0}$ at $x_0\in U$ is $\gamma < \infty$. Then there is $p > 1$, a ball $B_{\bar{r}}(x_0)$ about $x_0$, a neighbourhood $V_{\theta_0}$ of $\theta_0$, and a hyperplane $P \subset \R^{n+1}$ such that
\[
\{f^{\theta} = 0\} \cap B_{\bar{r}}(x_0) \subset \bigcup_{i=1}^{\gamma} \Gamma_{f^{\theta}_i} \qquad \text{for every}\ \theta \in V_{\theta_0}, 
\]
where the functions $f^{\theta}_i$ belong to $W^{1,p}(P,P^{\perp})$ for $\theta \in V_{\theta_0}$, $i=1,\dots,\gamma$ and
\[ \sup_{\substack{i= 1,\dots,\gamma,\\ \theta \in V_{\theta_0}}} \norm{f^{\theta}_i}_{W^{1,p}}  < \infty.\] 
\end{Thm}

Our proof of Theorem \ref{T:Main}, which does not rely on Proposition \ref{P:Bar}, is given in \S \ref{S:ProofMain}. The main technical input is the work of Par\'usinski--Rainer \cite[Theorem 3.5]{parusinskirainerregularity} on the regularity of roots of smooth families of polynomials (see Theorem \ref{T:ParusinskiRainer} below). The $W^{1,p}$ regularity is optimal if one is given a continuous parametrization of the nodal set of a smooth function over a given hyperplane (e.g. $f(x,y) = y^q - x$ for some $q > 1$ and the hyperplane $\{y=0\}$). It is possible that one could make a  `good' choice of hyperplane and establish better regularity of the functions $f^{\theta}_i$. 

 \subsection{Nodal sets as p-sweepouts}\label{S:SweepoutApp}
As part of an analogy suggested by Gromov \cite{gromovdimension} between the min-max $p$-widths $\omega_p(M)$ of a compact Riemannian manifold $(M,g)$ (see \cite[Definition 4.3]{marquesnevesexistence} or \cite[Appendix 3]{guthminimax}) and the eigenvalues of the Laplacian $\Delta_g$, Marques and Neves proposed in \cite[Section 9]{marquesnevesexistence} studying $p$-sweepouts given by nodal sets of linear combinations of eigenfunctions. We show here that one can indeed construct admissible $p$-sweepouts in this way (in the sense of \cite[\textsection 4.2]{marquesnevesexistence}). In fact, we prove the following stronger result:

\begin{Thm}\label{T:Psweepout} Let $(M^{n+1},g)$ be a smooth, compact Riemannian manifold and suppose that $f_0,\dots,f_p \in C^{\infty}(M, \R)$ satisfy the following property: There exists $\gamma > 0$ such that for every $x_0 \in M$ and every $(\theta_0,\dots,\theta_p) \in \R^{p+1}\setminus\{0\}$, the vanishing order of $\theta_0f_0 + \dots + \theta_pf_p$ at $x_0$ is at most $\gamma$. Then the map
\begin{align*}
\Phi\ :\ \R\mathbb{P}^p\ &\to\ \mathcal{Z}_n(M,\mathbb{Z}_2)\\
[\theta_0:\dots:\theta_p]\ &\mapsto\ \partial\{\theta_0f_0  + \dots + \theta_pf_p< 0\}
\end{align*}
is an admissible $p$-sweepout.
\end{Thm}

\noindent Here, $\mathcal{Z}_n(M,\mathbb{Z}_2)$ is the space of mod 2 flat $n$-cycles in $M$ (see \cite[p. 423]{federergmt}). For the proof see \S \ref{S:SweepoutProof}. Let us write 
\[\Delta_g \varphi_j = \leb_j^2 \varphi_j,\qquad 0=\leb_0<\leb_1<\leb_2\leq \cdots \nearrow \infty\]
for the eigenvalues and eigenfunctions of the Laplacian (with multiplicity). As mentioned above, that non-zero finite linear combinations of the eigenfunctions have finite vanishing order was first proved in \cite[Thm 4.2]{linnodal} and later by different methods in \cite[Thm 14.10]{jerisonlebeaunodal} and \cite[Thm 4.1]{donnellynodalsums}. Thus Theorem \ref{T:Psweepout} applies to linear combinations of the eigenfunctions $\varphi_j$. In the context of $p-$sweepouts, it is therefore natural to define
\[
\Upphi_p(M) := \sup_{\theta \in \R\mathbb{P}^p} \mathbf{M}(\partial\{\theta_0\varphi_0  + \dots + \theta_p\varphi_p< 0\}),
\]
where $\mathbf{M}$ denotes the mass of an element in $\mathcal Z_n(M, \Z/2)$. Combining the Weyl-type lower bounds on $\w_p(M)$ \cite[\S 4.2]{gromovdimension}, \cite[\S 3]{guthminimax} and Theorem \ref{T:Psweepout} gives
\begin{align} \label{E:ineqchain}
c\,p^{\tfrac{1}{n+1}} \leq \omega_p(M) \leq \Upphi_p(M) \leq \sup_{\theta \in \R\mathbb{P}^p} \mathcal{H}^n(\{\theta_0\varphi_0  + \dots + \theta_p\varphi_p = 0\}).
\end{align}
To see the last inequality, we use that if $f$ is a function of finite vanishing order, then $\mathbf{M}(\dell \set{f<0})$ is simply the Hausdorff measure $\mathcal H^n\lr{\dell \set{f<0}}$ of the topological boundary of $\set{f<0}.$ Notice that the linear combination of eigenfunctions $f(x_1,x_2)=1+\cos(x_1)$ on the two-torus $\mathbb T=\R^2/(2\pi \Z)^2$ satisfies 
\[\mathbf{M}(\partial \{f < 0\})=0<\hcal^1(\{f = 0\})=2\pi.\] 
That is, for a general linear combination of eigenfunctions the mass of the associated mod 2 flat chain can be strictly less than the measure of the nodal set because the nodal set can have a large singular part. However, it is not known if the third inequality in \eqref{E:ineqchain} can in fact be strict.

Marques and Neves also raise the question of understanding the exact asymptotic relationship between $\omega_p(M)$ and $\Upphi_p(M)$ as $p \to \infty$. Their ``asymptotic optimality'' conjecture is that $\Upphi_p(M)/\omega_p(M)$ tends to $1$. 

In the course of proving Theorem \ref{T:Psweepout}, we establish the following:

\begin{corollary}\label{C:Weakhomequiv} Let $(M^{n+1},g)$ be a smooth, compact Riemannian manifold and let $\{\varphi_j\}_{j=1}^{\infty}$ be an orthonormal basis for $L^2(M,g)$ consisting of real-valued eigenfunctions of the Laplacian. The map
\begin{align*}
\Phi\ :\ \R\mathbb{P}^{\infty}\ &\to\ \mathcal{Z}_n(M,\mathbb{Z}_2)\\
[\theta_0:\dots:\theta_p: 0 : 0: \cdots ]\ &\mapsto\ \partial\{\theta_0\varphi_0 + \theta_1\varphi_1 + \dots + \theta_p\varphi_p < 0\}
\end{align*}
is a weak homotopy equivalence.
\end{corollary}
Corollary \ref{C:Weakhomequiv} is proved in \S \ref{S:HomProof}. Finally, we provide a new proof, given in \S \ref{S:UBPf}, of the Weyl-type upper bounds on the min-max p-widths $\w_p(M)$ of a compact smooth manifold $M$, originally established by Guth in \cite[Thm 1]{guthminimax} when $M$ is a closed unit ball and for more general compact manifolds by Marques-Neves in \cite[Thm 5.1]{marquesnevesexistence}. Our argument is similar to the one outlined by Gromov in \cite[\S 4.2B]{gromovdimension}. 

 \begin{Thm}[\cite{gromovdimension, guthminimax, marquesnevesexistence}] \label{T:UpperBound} Let $M$ be a compact smooth manifold $M$ without boundary, and define the min-max p-width $\w_p(M)$ by
 \begin{align*}
 \w_p(M) = \inf_{\Phi \in \mathcal{P}_p} \sup_{x\in X} \mathbf{M}(\Phi(x)),
 \end{align*}
 where the infimum is over admissible $p$-sweepouts $\Phi: X \to \mathcal{Z}_n(M,\mathbb{Z}_2)$. Then, 
 \begin{align*}
 \w_p(M) \leq C\cdot p^{\frac{1}{n+1}}.
 \end{align*}
 \end{Thm}

\subsection{Nodal Sets Under Heat Flow} \label{S:Heat}  Given a function $v \in L^2(M)$, write 
\[v = \sum_{j=0}^{\infty}c_{j}\varphi_{j},\qquad c_{j}\in l^2(\mathbb{N}).\]
For each $\eps>0$, define $N_{\eps}(v) : L^{2}(M) \to [-\infty,\infty]$ by
\begin{align} \label{eqn:tildeN}
N_{\eps}(v) = \log\left(\sum_{j=0}^{\infty}c_j^2e^{\eps {\la_j}}\right).
\end{align}
It follows from both \cite[Theorem 4.3]{linnodal} and \cite[Theorem 14.10]{jerisonlebeaunodal} that if $N_{\eps}(v)$ is finite for some $\eps>0,$ then $v$ has finite vanishing order, bounded by an explicit function of $N_\eps(v)$. Therefore, Theorem \ref{T:Main} also applies to certain infinite linear combinations of eigenfunctions. Let, for instance, $u:M\x \R_+ \gives \R$ solve the heat equation 
\begin{align}
\left(\pa_{t} + \Delta_{g}\right)u(x,t) = 0\label{E:HE}
\end{align}
with initial data $u(x,0) = u_{0}(x)\in L^2(M)$. Suppose that $N_\ep(u_0)<\infty$ for some $\ep>0$.  Writing $\psi = \Pi_{\leb_k}u_0$ for the first non-zero eigenspace projection of $u_0,$ and changing the time variable from $t$ to $\theta=\frac{2}{\pi}\arctan(t),$ we define $f^{\theta}(x)$ by
\[ f^{\theta}(x) := \begin{cases}
e^{-\lambda_k^2\tan(\theta\pi/2)}u\lr{x,\tan(\theta\pi/2)}, &\qquad \theta\in [0,1)\\
\psi(x),& \qquad \theta =1
\end{cases}.\]
It follows from writing $f^{\theta}(x)$ as a Fourier series that $F(x,\theta):=f^{\theta}(x)\in C^\infty(M\x [0,1]).$ 
Setting $\Theta=[0,1],$ it is easy to see that $f^{\theta}(x)$  satisfies 
\begin{align}\label{E:FinFreq}
\sup_{\theta\in\Theta}\inf_{\eps>0} N_{\eps}(f^{\theta}) <\infty,
\end{align}
 since $N_\ep(u_0)<\infty$ for some $\ep>0.$ Therefore, there exists $C>0$ so that
\begin{align}
\sup_{t\geq 0}\hcal^{n}(\{x\in M: u(x,t) = 0\}) \leq C. 
\end{align}
It is natural to compare the nodal set $Z_\theta = \set{f^\theta(x)=0}$ as $\theta \gives 1$ with the nodal set of $\psi(x)=\lim_{\theta \gives 1} f^{\theta}(x).$ We do this with the help of Corollary \ref{C:Cont}, which follows from either Theorem \ref{T:Main} or Proposition \ref{P:Bar}. We write
\[ \text{Sing}_f := \setst{x\in U}{f(x)=0,\, \nabla f(x)=0}\]
for the singular set of a smooth function. 
\begin{corollary}\label{C:Cont}
Let $U$ be an open subset of $\R^{n+1}$. Fix $f \in C^{\infty}(U,\R)$ with finite vanishing order on $U.$ Suppose $F\in C^{\infty}(U\times[0,1])$ with $f^{\theta}(\cdot) := F(\cdot,\theta)$ for $\theta\in[0,1]$ and $f^{1} = f$. Then, for any compact $K \subset U$ with $\hcal^n(K\cap \emph{Sing}_f) = 0$, we have 
\[ \lim_{\theta\to1} \hcal^n(Z_{f^\theta} \cap K) =  \hcal^n(Z_f \cap K).\]
\end{corollary}
Corollary \ref{C:Cont}, which we will prove in \S \ref{S:nonconc} follows from the implicit function theorem if $\text{Sing}_f=\emptyset$ but otherwise is non-trivial. In general, $\hcal^n(Z_f)$ is neither lower nor upper semi-continuous as a function of $f$. 

Corollary \ref{C:Cont} applies to the function $u(x,t)$ satisfying \eqref{E:HE}. Indeed, note that by \cite[Thm 1.7]{hardtsimonnodal}, since $\psi$ is an eigenfunction, we have $\hcal^{n-1}(\text{Sing}_\psi) < \infty$. Hence,
\begin{equation}
\lim_{t\gives \infty}\hcal^n(Z_{u(\cdot, t)})=\hcal^n(Z_\psi)\label{E:HeatLimit}.
\end{equation}

\subsection*{Acknowledgements} The authors would like to thank Larry Guth for a series of discussions that led them to these problems and Steve Zelditch for helpful remarks about an earlier version of this article.

\section{$W^{1,p}$ Regularity for Nodal Sets} \label{S:W1p}
 In this section, we prove Theorem \ref{T:Main}. We begin by recalling some results and outlining the proof in \S \ref{S:Outline}. We give the full argument in \S \ref{S:ProofMain}.

\subsection{Outline of the Proof of Theorem \ref{T:Main} and Background}\label{S:Outline}
Our proof of Theorem \ref{T:Main} has three steps. The first is to apply the Malgrange Preparation Theorem (\cite{malgrangepreparation}; or see \cite{nirenbergproofmalgrange} or \cite[Chapter IV \textsection 2]{golubitskyguilleminstable} for later proofs). 

\begin{Thm}[Malgrange; \cite{malgrangepreparation}] \label{T:Malgrange} Let $U$ be an open subset in $\R^{n+1}$ and suppose that $f \in C^{\infty}(U)$ satisfies
\be
\frac{\partial^j}{\partial x_{n+1}^j}f(0) = 0\quad \forall j \leq k-1\qquad \text{and}\qquad \frac{\partial^k}{\partial x_{n+1}^k}f(0) \neq 0.
\ee
Then there exists an open neighborhood $\tilde{U}$ of $0$, a non-vanishing smooth function $c \in C^{\infty}(\tilde{U})$ and smooth functions $a_j \in C^{\infty}(\{x_{n+1} = 0\} \cap \tilde{U})$ for $j=0,\dots,k-1$, such that (writing $\bar{x} = (x_1,\dots,x_n,0)$) 
 we have
\be
f(x) = c(x)\lr{x_{n+1}^k + a_{k-1}(\bar{x})x_{n+1}^{k-1} + \cdots + a_0(\bar{x})}
\ee
in $\tilde{U}$.
\end{Thm}

This theorem (which is also used in the proof of Proposition \ref{P:Bar}) will allow us to deduce that close to a point at which a smooth function has finite vanishing order, the nodal set is described by the real roots of a smooth family of polynomials. The second step in our proof comes from the work of De Lellis--Grisanti--Tilli \cite{delellisgrisantitilliregular} about continuous selections of $Q-$valued functions: 

\begin{proposition}[Theorem 1.2 of \cite{delellisgrisantitilliregular}]\label{P:DLT} Let $f:[a,b]\gives \mathcal Q_q(\R^n)$ be a $C^{k,\alpha}$ $Q$-valued function. Then there exist functions $g_i:[a,b]\gives \R^n$ such that $g_i\in C^{k,\alpha}([a,b])$ and the $Q-$tuple $\set{f(x)}$ coincides with $\set{g_i(x)}_{i=1}^Q$ for every $x.$
\end{proposition}

\begin{proposition}[Theorem 5.1 of \cite{delellisgrisantitilliregular}]\label{P:DLT2} Let $A\subseteq \R^m.$ If $f:A\gives \mathcal Q_Q(\R)$ is continuous, then there exist continuous functions $g_i:A\gives \R$ for $1\leq i \leq Q$ such that $f(x)=\sum_{i=1}^Q [[g_i(x)]]$
\end{proposition}

Given such a continuous selection, the third step, which is the key technical ingredient to our argument, is the recent work \cite{parusinskirainerregularity} of Parusi\'{n}ski and Rainer on the regularity of a continuous parametrization of the roots for such a family.

\begin{Thm}[Theorem 3.5 of \cite{parusinskirainerregularity}] \label{T:ParusinskiRainer}
Fix $k \in \mathbb{N}$. There exists $p = p(k) > 1$ such that the following is true.  Let $I\subseteq \R$ be a compact interval and let $\{ P_{a_\nu}\}_{\nu \in \mathcal N}$, for some indexing set $\mathcal{N}$, denote a family of monic polynomials 
\be
P_{a_{\nu}(t)}(X) := X^k + a_{\nu,k-1}(t)X^{k-1} + \dots + a_{\nu,1}(t)X + a_{\nu,0}(t).
\ee
with $a_{\nu,j} \in C^{\infty}(I;\C)$ for all $\nu \in \mathcal{N}$, $j = 0,\dots,k-1$. Let 
\be 
\Xi := \{\lambda_{\nu} \in C^0(I;\C) : P_{a_{\nu}}(\lambda_{\nu}) = 0\ \text{on}\ I\ \text{for some}\ \nu \in \mathcal{N} \}.
\ee
Then, the distributional derivative of each $\leb_j$ is a measurable function on $I$ with $\leb_j'\in L^q(I)$ for every $q\in [1,p).$ and if $\{a_{\nu,j}\}_{j=0,\dots,k-1;\ \nu \in \mathcal{N}}$ is bounded in $C^L(I;\C)$ for some sufficiently large $L$, then $\Xi$ is bounded in $W^{1,q}(I;\mathbb{C})$ for every $q \in [1,p)$.
\end{Thm}



\subsection{Proof of Theorem \ref{T:Main}}\label{S:ProofMain} The following Lemma reduces Theorem \ref{T:Main} to a local statement in which we can apply the regularity of roots result given in Theorem \ref{T:ParusinskiRainer}.
\begin{Lem}[Reduction to polynomials with smoothly varying coefficients]\label{L:Reduction}
Let $K\subseteq U$ be compact. There exist $R,\bar{r}>0$, finitely many points $(x_i,\theta_i)\in K\x \Theta$, as well as coordinate patches $U_i=\set{y_1,\ldots, y_n, t}$ and $V_i$ centered at $x_i,\theta_i$ with the following property. For every $(x,\theta)\in K\x \Theta,$ either $Z_{f^\theta}\cap B(x,\bar{r})=\emptyset$ for every $\theta$ or there exists $i$ so that for every $\rho\in (0,\bar{r})$
\begin{equation}
B_{\rho}(x) \subseteq \mathrm{C}, \label{E:rbarprop}
\end{equation}
where $\mathrm{C} =(-R,R)^{n+1}$ is an open cube centered at the origin in $U_i.$ Moreover, in each coordinate patch $U_i\x V_i,$ there exists $Q\leq \gamma$ and smooth functions $a_q(y,\theta)$ so that
\[Z_{f^\theta}|_{U_i\x V_i}=Z_{P^\theta}\]
with
\begin{equation}
P^\theta(y,t)=t^Q+\sum_{q=0}^{Q-1}t^q a_q(y,\theta).\label{E:Main3}
\end{equation}
\end{Lem}
\begin{proof}
Write $Z_K=\setst{(x,\theta)\in K\x \Theta}{f(x,\theta)=0}.$ For every $(x,\theta)\in Z_K,$ we combine the finite vanishing assumption on $f^\theta$ with Malgrange preparation (Theorem \ref{T:Malgrange}). This yields the existence of $r=r(x,\theta)>0,$ $Q \leq \gamma$, coordinates patches $U=U(x,\theta)=\set{y_1,\ldots, y_n, t,\theta},\,V=V(x,\theta)$ centered at $x,\theta,$ and smooth functions 
\[a_q: (-R,R)^n\x V\gives \R,\qquad  c: (-R,R)^{n+1}\x V\to \R\]
in these coordinates so that $c$ is non-vanishing and
\begin{align} \label{E:Main2}
f(y,t,\theta) = c(y,t,\theta) \cdot P^{\theta}(y,t)
\end{align}
with $P^\theta$ as in \eqref{E:Main3}. This means that the zero set of $f^\theta$ restricted to $\mathrm{C}\x V$ coincides with that of $P^\theta.$ The proof is completed by applying the Lebesgue number lemma to the covering of $\pi(Z_K)$ by the collection of coordinate cubes $\mathrm{C},$ where $\pi:Z_K\to K$ is the natural projection.   
\end{proof}

Lemma \ref{L:Reduction} reduces Theorem \ref{T:Main} to the case when $f^\theta=P^\theta$ and the set $U$ is $\mathrm{C}=(-R,R)^{n+1}$. Let us denote $\Omega := \mathrm{C} \cap \{t = 0\}$. The $Q$ complex roots of a degree $Q$ polynomial depend continuously on the coefficients, which means that there exists a continuous $Q$-valued function $R \in C^0(\W\x \Theta, \mathcal A_Q(\C))$ such that
\begin{equation}
 \label{E:Main6}
R(y,\theta) = \sum_{t : P(y,t,\theta)  = 0} \lbr t \rbr.
\end{equation}
We will write $R^{\theta}(y) = R(y,\theta)$ and define $\Rcal^\theta(y) := \Re\lr{R^\theta(y)}.$ By Proposition \ref{P:DLT2}, there exist continuous single-valued functions $\Rcal_j^\theta(y) : \W \x \Theta \to \{t=0\}^{\perp} \simeq \R$ for $1\leq j \leq Q,$ with the property that for every $y \in \Omega$ we have $\Rcal^{\theta}(y) = \sum_{j=1}^Q \lbr \Rcal_j^{\theta}(y) \rbr$. Hence, 
\be \label{E:Main7}
Z_{f^{\theta}} \subseteq \bigcup_{j=1}^Q \Gamma_{\Rcal_j^{\theta}},
\ee
where $\Gamma_g$ denotes the graph of $g.$ 
We now check that each $R_j^\theta$ belongs to $W^{1,p}(\W)$ for some $p>1.$ 

To see this, fix $i \in \{1,\dots,n\}$ and let $\mathcal L_i$ denote the set of lines parallel to the $y_i$-axis that intersect $\Omega$. Proposition \ref{P:DLT} implies that for any line $L \in \mathcal{L}_i$, there exist continuous functions $R_{j,L}^\theta\in C^0((L \cap \Omega) \times V, \C)$ for $j=1,\ldots, Q$ such that for every $(y,\theta) \in (L \cap \Omega) \times V$ we have
\begin{equation}
 \label{E:Main8}
\sum_{j=1}^Q \lbr R_{j,L}^\theta(y) \rbr  = \sum_{t : P(y,t,\theta) = 0} \lbr z \rbr.
\end{equation}
In order to apply Theorem \ref{T:ParusinskiRainer}, set $\mathcal{N} = V\times \mathcal{L}_i$, $I := [-R,R]$ and define $\tilde{a}_{(\theta,L),j} \in C^{\infty}(I,\C)$ to be the restriction of $a_j(\cdot, \theta)$ to $L$: 
\[\tilde{a}_{(\theta,L),j}(s) = a_j(se_i,\theta),\] 
where $e_i$ is the $i$th standard basis vector. Notice that $\{\tilde{a}_{(\theta,L),j}\}_{(\theta,L) \in \mathcal{N};\ j=0,\dots,Q-1}$ is bounded in $C^k(I,\C)$ for every $k$. Thus by Theorem \ref{T:ParusinskiRainer} and the fact that $i$ was arbitrary, there exists $p > 1$ and a constant $C > 0$ such that
\begin{equation}
\label{E:Main9}
\sup_{\substack{1\leq i \leq n,1 \leq  j\leq Q\\ \theta \in V,\, L\in \mathcal L_i} } \norm{\dell_{x_i}R_{j,L}^\theta}_{L^{p}(L\cap \Omega)} < C.
\end{equation}
The same therefore holds with $R_{j,L}^\theta$ replaced by its real part. Hence,
\be \label{E:Main10}
\norm{\dell_{x_i}\Rcal_j^\theta}_{L^p(L\cap \Omega)}\leq \sum_{j=1}^Q \norm{\dell_{x_i}\Rcal_j^\theta}_{L^p(L\cap \Omega)} = \sum_{j=1}^Q \norm{\dell_{x_i}\Re \lr{ R_{j,L}^\theta}}_{L^p(L\cap \Omega)}
\ee
since for every $y\in L\cap \Omega$ the $Q$-tuple $\lr{\Re\lr{R_{1,L}^\theta(y)},\ldots, R_{Q,L}^\theta(y)}$ is a permutation of $\lr{\Rcal_1^\theta(y),\ldots, R_{Q}^\theta(y)}$. Combining this with \eqref{E:Main9} and Fubini's Theorem, we deduce that
\begin{equation}
\label{E:Main11}
\sup_{\substack{1\leq i \leq n, 1 \leq j\leq Q\\ \theta \in V}}\int_{\Omega} \abs{\dell_{x_i} \Rcal_j^\theta(x)}^p dx<\infty.
\end{equation}
This shows that $\Rcal_j^\theta\in W^{1,p}(\W\x \Theta)$ and completes the proof of Theorem \ref{T:Main}. \qed

\section{Non-concentration of nodal sets and proof of Corollary \ref{C:Cont}}  \label{S:nonconc}
The following result gives an estimate on the extent to which the nodal set of a smooth function can concentrate near a lower dimensional set.
\begin{proposition}\label{T:NonConcentration} Let $U$ be an open subset of $\R^{n+1}$ and let $\Theta$ be a smooth, compact manifold, possibly with boundary. Consider $F\in C^\infty(U\x \Theta),$ and suppose there exists $\gamma > 0$ such that for every $\theta \in \Theta $ and $x \in U$, the vanishing order of $f^\theta:=F\lr{\cdot, \theta}$ at $x$ is at most $\gamma$. Fix compact sets $K\subset U$ and $E\subseteq \R^{n+1}$ with $E$ being $m-$rectifiable for some $m\leq n.$ Write 
\[\emph{Lip}_1\lr{E,K}:=\setst{\iota (E)}{\iota:E\gives K~\text{ is Lipschitz with }~ \norm{\iota}_{Lip}\leq 1}\]
and denote by $A_r$ the $r-$neighborhood of $A\subseteq \R^{n+1}.$ Then there exist $\bar{r}>0$ and $C = C(n) > 0$, so that the following non-concentration estimate holds:
\begin{equation}
\sup_{\substack{E'\in \emph{Lip}_1(E,K)\\ \theta \in \Theta}} \hcal^n\lr{Z_{f^\theta} \, \cap \,E_r'\,\cap\, K }\leq Cr^{-1}\cdot \hcal^{n+1}(E_r) \qquad \forall r\leq \bar{r}.\label{E:NonConcentrationE}
\end{equation}
\end{proposition}

Proposition \ref{T:NonConcentration} follows easily from Proposition \ref{P:Bar} and the fact that for a closed $m-$rectifiable $E\subseteq \R^{n+1}$ we have
\[\lim_{r\gives 0} \frac{\mathcal H^{n+1}\lr{E_r}}{r^{n+1-m}}=\frac{\alpha(n+1-m)}{2^m\alpha(m)} \mathcal H^m(E),\]
where $\alpha(l)$ is the volume of a unit ball in $\R^l$  (\cite[Thm. 3.2.29]{federergmt}). Using Theorem \ref{T:Main}, rather than Proposition \ref{P:Bar}, one can prove a weaker version of \eqref{E:NonConcentrationE} in which the constant $C$ is allowed to depend on $f$ and the expression $(r^{-1}\cdot \hcal^{n+1}(E_r))$ is raised to some power $\delta > 0$. Proposition \ref{T:NonConcentration} will be used in \S \ref{S:PSweepout} to check a non-concentration condition in the definition of a p-sweepout. We use it now to prove the continuity result in Corollary \ref{C:Cont}.



\subsection*{Proof of Corollary \ref{C:Cont}} Let us write $E=\text{Sing}_{f}$, and let $K\subset U$ be a fixed compact set. By the implicit function theorem, for every compact subset $L\subseteq K\backslash E$ and every $\ep>0$ there exists $\eta>0$ so that 
\begin{align} \label{E:IFT1}
\abs{\hcal^n\lr{Z_{f_{\theta}}\cap L}-\hcal^n(Z_{f}\cap L)}\leq \ep,\qquad \forall \theta\geq 1-\eta.
\end{align}
Moreover, by Proposition \ref{P:Bar}, for $r$ sufficiently small, we have the estimate
\begin{align*}
\sup_{\theta\in[1-\eta,1]}\hcal^{n}\left(Z_{f_{\theta}}\cap E_{r}\cap K\right) \leq C r^{-1}\cdot \hcal^{n+1}(E_r).
\end{align*}
Since $E$ is a closed $n-$rectifiable set, its $n$-dimensional Minkowski content is equal to a constant times its $n-$dimensional Hausdorff:
\begin{align*}
\lim_{r\to0} \frac{\hcal^{n+1}(E_r)}{r} = \hcal^n(E)=0.
\end{align*}
In particular, for $r>0$ sufficiently small
\begin{align*}
\sup_{\theta\in[1-\eta,1]}\hcal^{n}\left(Z_{f_{\theta}}\cap E_{r}\cap K\right) \leq \ep.
\end{align*}
Combining this with the estimate in \eqref{E:IFT1} completes the proof of the Corollary.  \qed

\section{Nodal Sets as $p$-Sweepouts}\label{S:PSweepout} 
In this section we will prove Theorem \ref{T:Psweepout} and Corollary \ref{C:Weakhomequiv} together with Theorem \ref{T:UpperBound}. We will need the following simple fact.

\begin{Lem} \label{L:Continflat}
Let $(M^{n+1},g)$ be a smooth Riemannian manifold without boundary. Suppose that $f \in C^0(M,\R)$ has $\mathcal{H}^{n+1}(Z_f) = 0$ and fix $\varphi \in L^{\infty}(M)$. Then the map 
\begin{align*}
\R &\quad \rightarrow \quad  \mathcal Z_n(M,\Z_2) \\
\delta &\quad \mapsto\quad  \dell\abs{\set{f+\delta \phi < 0}}
\end{align*}
is continuous at $\delta = 0$ with respect to the flat topology on $\mathcal Z_n(M,\Z_2).$
\end{Lem}
\begin{proof}
Fix $\ep > 0$ and write $\mathcal F(S,T)$ for the distance between $S,T\in \mathcal Z_n(M,\Z_2)$ in the flat metric (see \cite[p. 367]{federergmt}). The definition of the flat metric implies that 
\begin{equation} \label{E:Continflat1}
\mathcal{F}\lr{\dell\abs{\set{f+\delta \phi < 0}},\dell\abs{\set{f < 0}}} \leq \mathcal{H}^{n+1}(\{f < 0 < f + \delta\varphi\}). 
\end{equation}
Using the definition of Hausdorff measure, the compactness of $Z_f$ and the Lebesgue number lemma, there exists $\alpha_0 = \alpha_0(\ep) > 0$ such that
\begin{equation} \label{E:Continflat2}
\hcal^{n+1}\lr{ \{x \in M\ :\ d(x, Z_f) < \alpha \} } < \ep,\qquad \forall \alpha \leq \alpha_0.
\end{equation}
Since $f$ is uniformly continuous and 
\begin{equation} \label{E:Continflat3}
\{f < 0 < f + \delta\varphi\} \subset \{ |f| \leq \delta\norm{\varphi}_{L^{\infty}(M)}\},
\end{equation}
there exists $\delta_0 = \delta_0(\ep) > 0$ such that
\begin{equation} \label{E:Continflat4}
\{f < 0 < f + \delta\varphi\}  \subset  \{x \in M\ :\ d(x, Z_f) < \alpha_0/2 \} , \qquad \forall \delta \leq \delta_0. 
\end{equation}
Thus choosing $\delta < \delta_0$ shows that the left-hand side of \eqref{E:Continflat1} is at most $\ep$, which completes the proof.
\end{proof}

\subsection{$p$-Sweepouts} \label{SS:Psweepout} Let us recall the definition of a $p$-sweepout (see \cite[\textsection 3.7, \textsection 4.1]{marquesnevesexistence}). Firstly, a map $\Phi : S^1 \gives \mathcal Z_n\lr{M,\Z_2}$ is a \textit{sweepout} if it is continuous in the flat topology and the class $[\Phi] \in \pi_1(Z_n\lr{M,\Z_2})$ is non-zero. If we let $X$ denote a cubical subcomplex of $[0,1]^m$ for some $m$, then a continuous map $\Phi : X \gives \mathcal Z_n\lr{M,\Z_2}$ is an \textit{admissible $p$-sweepout} if there exists $\lambda\in H^1\lr{X,\Z_2}$ such that
  \begin{enumerate}[label=(\roman*), ref=(\roman*)]
  \item \label{psweepout1} For any $\gamma: S^1\gives X$ we have $\lambda\lr{\gamma}\neq 0$ if and only if $\Phi\circ \gamma:S^1\gives \mathcal Z_n\lr{M,\Z_2}$ is a sweepout;
  \item \label{psweepout2} The cup product $\lambda^p\neq 0\in H^p\lr{X,\Z_2}.$ 
  \item \label{psweepout3} With $B_r(p)$ denoting the ball of radius $r$ centered at $p$ in $M,$ we have
  \be 
  \limsup_{r\gives 0^+}\sup_{x\in X,\,\, p\in M} \norm{\Phi(x)}(B_r(p))=0.
  \ee 
  \end{enumerate}
  
\begin{remark} \label{R:Sweepout} We recall the content of \cite[Remark 4.2]{marquesnevesexistence} which says that if $\gamma$ and $\gamma'$ are homotopic in $X$, then $\Phi \circ \gamma$ is a sweepout if and only if $\Phi \circ \gamma'$ is a sweepout. \end{remark}

\subsection{Almgren's Isomorphism} \label{S:Almgreniso} In \cite{almgrenhomotopy}, Almgren constructed an isomorphism between $\pi_1(\mathcal{Z}_n(M,\Z_2))$ and $H_{n+1}(M,\Z_2)$. For the proof of Theorem \ref{T:Psweepout} we will need to know how to use Almgren's isomorphism to check when an element of $\pi_1(\mathcal{Z}_n(M,\Z_2))$ is non-zero (so we recall here just the essentials that are required to do that and refer the reader to \cite[\S 3]{marquesnevesexistence} or to the original paper \cite{almgrenhomotopy} for more information). Given a continuous map $\Phi : S^1 \to \mathcal{Z}_n(M,\Z_2)$, there exist $0 = s_0 < s_1 < \dots < s_K = 2\pi$, a constant $\rho = \rho(M) \geq 1$ and $A_j \in \mathcal{Z}_{n+1}(M,\Z_2)$ for $j=0,\dots,K-1$ such that 
\[
\dell A_j = \Phi(s_{j+1}) - \Phi(s_j), \qquad \mathbf{M}(A_j) \leq \rho \mathcal{F}(\Phi(s_{j+1}),\Phi(s_j)),
\]
and such that $\bigl[\sum_{j=0}^{K-1}A_j \bigr] \in H_{n+1}(M,\Z_2)$ only depends on the homotopy class of $\Phi$ (to see in general that $\sum_{j=0}^{K-1}A_j$ defines an element of $H_{n+1}(M,\Z_2)$, see \cite[\S 4.4.6]{federergmt}). Thus we may define
\be 
F_M(\Phi) := \Biggl[\sum_{j=0}^{K-1}A_j \Biggr] \in H_{n+1}(M,\Z_2).
\ee
The induced map $F_M : \pi_1(\mathcal{Z}_n(M,\Z_2)) \to H_{n+1}(M,\Z_2)$ is well-defined and an isomorphism. Moreover, the $A_j$ are unique in the following sense: There is a constant $\nu = \nu(M) > 0$ such that if $B_j \in \mathcal{Z}_{n+1}(M,\Z_2)$ for $j=0,\dots,K-1$ are such that $\mathbf{M}(B_j) \leq \nu$ and $\dell B_j =  \Phi(s_{j+1}) - \Phi(s_j)$, then $A_j = B_j$.

\subsection{Proof of Theorem \ref{T:Psweepout}}\label{S:SweepoutProof} For $\theta = [\theta_0 : \cdots : \theta_p] \in \R\mathbb{P}^p$ and $x \in M$, write $f^{\theta}(x) = \theta_0f_0(x) + \dots \theta_pf_p(x)$. By Theorem \ref{T:Main}, 
\[
\mathcal{H}^{n+1}(Z_{f^{\theta}}) = 0\qquad  \forall \theta \in \R\mathbb{P}^p.
\] 
Lemma \ref{L:Continflat} thus implies that $\Phi$ is continuous in the flat topology. The non-concentration estimate in Proposition \ref{T:NonConcentration}  also shows that $\Phi$ satisfies \ref{psweepout3}. Moreover, since $X$ is homeomorphic to $\R \mathbb P^p$ in our case, we know that $H^1(X,\Z_2) = H^p(X,\Z_2) = \Z_2$. This means that the generator $\lambda$ of $H^1(X,\Z_2)$ satisfies $\lambda^p \neq 0$ in $H^p(X,\Z_2)$, which shows that $\Phi$ satisfies \ref{psweepout2}. It therefore remains to check \ref{psweepout1} for which we need the following:

\begin{claim}
There exists a generator $\hat{\gamma}$ of $\pi_1(\R \mathbb{P}^p) = \Z_2$ for which $\Phi \circ \hat{\gamma} \neq 0 \in \pi_1(Z_n\lr{M,\Z_2})$.
\end{claim}

\noindent Assuming this for the moment, we will prove \ref{psweepout1}. Let $\gamma : S^1 \to X$ be a continuous map. Note that since $X$ is homeomorphic to $\R\mathbb{P}^p$, this defines an element $[\gamma] \in \pi_1(\R\mathbb{P}^p)$. Now, if $\lambda(\gamma)\neq 0$, then $[\gamma] \neq 0$, which means that $[\gamma] = [\hat{\gamma}]$. Using Remark \ref{R:Sweepout} followed by the claim, this implies that $\Phi_*([\gamma])=\Phi_*([\hat{\gamma}]) \neq 0$, \emph{i.e.} $\Phi\circ\gamma$ is a sweepout. Conversely, if $\Phi\circ\gamma$ is a sweepout, then it must be the case that $[\gamma] \neq 0$, which implies that $\leb(\gamma) \neq 0$.

To prove the claim, consider the continuous map $\hat{\gamma} : S^1\gives \R\mathbb{P}^p$ given by 
\begin{equation*}
\hat{\gamma}(s)=[\cos(s/2),\sin(s/2):0:\cdots :0].
\end{equation*}
Therefore Almgren's isomorphism (Section \ref{S:Almgreniso}) implies that there exist $0 = s_0 < s_1 < \cdots <  s_K = 2\pi$ such that the class $[\Phi \circ \hat{\gamma}]$ is non-zero in $\pi_1\lr{\mathcal Z_n\lr{M,\Z_2}}$ if and only if
\begin{equation}
\sum_{j=0}^{K-1} \Biggl[ \set{p \in M\setminus Z_{f_0} : -\cot(s_j/2)< \frac{f_1}{f_0} < -\cot(s_{j+1}/2)} \Biggr] \in H_{n+1}\lr{M,\Z_2} \label{E:AlmgrenSum}
\end{equation}
Since we know that $\mathcal{H}^{n+1}(Z_{f^{\hat{\gamma}(s)}}) = 0$ for every $s \in S_1$ (by Theorem \ref{T:Main}), the sum above is equal to $[M]$,  which generates $H_{n+1}(M,\Z_2)$ and is therefore non-zero. This completes the proof of Theorem \ref{T:Psweepout}. \qed
\phantom\qedhere

\medskip

\subsection{Proof of Corollary \ref{C:Weakhomequiv}}\label{S:HomProof}  It can be shown that $\R\mathbb{P}^{\infty}$ and $\mathcal{Z}_n(M,\Z_2)$ are both weakly homotopically equivalent to the Eilenberg-MacLane space $K(\Z_2,1)$ (this means that they are connected, with $\pi_1 \simeq \Z_2$ and $\pi_k = 0$ for $k > 1$). So to establish a weak homotopy equivalence we only need to establish firstly continuity of the map (which follows from the previous arguments), and secondly that a generator of $\pi_1(\R\mathbb{P}^{\infty})$ is mapped to a generator of $\pi_1(\mathcal{Z}_n(M,\Z_2))$. But this is exactly what the argument above shows: We can pick $\hat{\gamma}$ as our generator of $\pi_1(\R\mathbb{P}^{\infty}) $ and then using the Almgren isomorphism we see that its image in $\pi_1(\mathcal{Z}_n(M,\Z_2))$ is non-trivial; and every non-trivial element is a generator. \qed
\\
\\
\subsection{Proof of Theorem \ref{T:UpperBound}}\label{S:UBPf} Fix a compact smooth manifold $M.$ Classical theorems of Whitney \cite[Theorems 1, 4]{whitneydifferentiable} guarantee the existence of a smooth diffeomorphism $\mathcal J:M\gives N$ between $M$ and a real analytic submanifold $N$ of Euclidean space, which of course admits a real analytic metric. 

Denote by $V_p$ the span of the first $p+1$ eigenfunctions of the Laplacian for this real analytic metric on $N.$ By \cite[Theorem 14.3]{jerisonlebeaunodal},
\begin{align} \label{E:UpperBound1}
\sup_{f\in V_p\backslash\{0\}} \hcal^{n}(Z_{f}) \leq Cp^{\frac{1}{n+1}},
\end{align}
where $C$ depends only on $N$. Moreover, Theorem \ref{T:Psweepout} shows that
\begin{align*}
\Phi_p\ :\ \R\mathbb{P}^p\ &\to\ \mathcal{Z}_n(N,\mathbb{Z}_2)\\
[\theta_0:\dots:\theta_p]\ &\mapsto\ \partial\{\theta_0\psi_0  + \dots + \theta_p\psi_p< 0\}
\end{align*}
is an admissible $p$-sweepout for all $p$. Composing $\Phi_p$ with the pullback $\mathcal{J}^{\ast}$ gives an admissible $p$-sweepout on $M$. Since
\begin{align*}
\mathbf{M}\left(\dell\{f<0\}\right) \leq \hcal^{n}(Z_{f})
\end{align*}
for all $f\in V_p\backslash\{0\}$, the estimate in \eqref{E:UpperBound1} completes the proof of the theorem. \qed

  \bibliographystyle{amsalpha}
 \bibliography{Referencesgeompde}

\end{document}